\documentclass[1p]{elsarticle}
\usepackage{color}
\usepackage{amsmath}
\usepackage{amssymb}
\usepackage{amsthm}
\usepackage[ruled ]{algorithm2e}

\def \vec#1{{\bf{#1}}}
\newcommand{\curl}{\vec{\nabla}\times}
\newcommand{\mydiv}{\vec{\nabla}\cdot}
\newcommand{\mygrad}{\vec{\nabla}}
\newcommand{\be}{\begin{equation}}
\newcommand{\ee}{\end{equation}}
\def\pd#1#2{\dfrac{\partial #1}{\partial #2}}
\def \mat#1{\underline{\underline{#1}}}

\newcommand{\bi}{\begin{itemize}}
\newcommand{\ei}{\end{itemize}}

\newcommand{\vb}{\mbox{\boldmath$\beta$}}
\newtheorem{thm}{Theorem}[section]

\newtheorem{lem}[thm]{Lemma}
\newenvironment{rem}[1][Remark]{\begin{trivlist}
\item[\hskip \labelsep {\bfseries #1}]}{\end{trivlist}}

\journal{Computational Physics}

\begin{document}
\begin{frontmatter}
\title{First-Order System Least Squares and the Energetic Variational Approach for Two-Phase Flow}

\author[label1]{J. H. Adler\corref{cor1}}
\ead{adler@math.psu.edu}
\address[label1]{Mathematics Department, Pennsylvania State University,
University Park, PA 16802}
\author[label1]{J. Brannick}
\ead{brannick@math.psu.edu}
\author[label1]{C. Liu}
\ead{liu@math.psu.edu}
\author[label2]{T. Manteuffel}
\ead{tmanteuf@colorado.edu}
\address[label2]{Department of Applied Mathematics, University of Colorado at Boulder, Boulder, CO 80309-0526}
\author[label1]{L. Zikatanov}
\ead{ltz@math.psu.edu}
\cortext[cor1]{Corresponding author}

\begin{abstract}
This paper develops a first-order system least-squares (FOSLS) formulation for equations of two-phase flow.  The main goal is to show that this discretization, along with numerical techniques such as nested iteration, algebraic multigrid, and adaptive local refinement, can be used to solve these types of complex fluid flow problems.  In addition, from an energetic variational approach, it can be shown that an important quantity to preserve in a given simulation is the energy law.  We discuss the energy law and inherent structure for two-phase flow using the Allen-Cahn interface model and indicate how it is related to other complex fluid models, such as magnetohydrodynamics.   Finally, we show that, using the FOSLS framework, one can still satisfy the appropriate energy law globally while using well-known numerical techniques.
\end{abstract}


\end{frontmatter}

\section{Introduction}\label{sec1}

\looseness=-1 Complex fluids involve multi-physics and multi-scale phenomena that require advanced techniques in scientific computing in order to be solved efficiently.  The goal of this paper is to present numerical techniques that have been well tested and shown to give accurate solutions with a reasonable amount of computational cost, while also preserving the global energy of the system.  While the systems can be described with partial differential equations (PDEs), in this case a time-dependent and nonlinear system of equations, an underlying energy law also describes the physics.  It is from these laws that the PDE is derived.  Therefore, approximate solutions to complex fluid problems should adequately approximate this energy law.  

Here, we look at one type of complex fluid problem, two-phase flow. The conventional sharp interface description of the mixture of two Newtonian fluids with density, $\rho_i$, and viscosity, $\mu_i$, assumes the mixture occupies the overall domain, $\Omega = \Omega_1 \cup \Omega_2$, as in figure \ref{fluidpic}.  The interface between two materials is the free moving interface, $\Gamma_t$.

\begin{figure}[h!]
\centering
\includegraphics[scale=0.5]{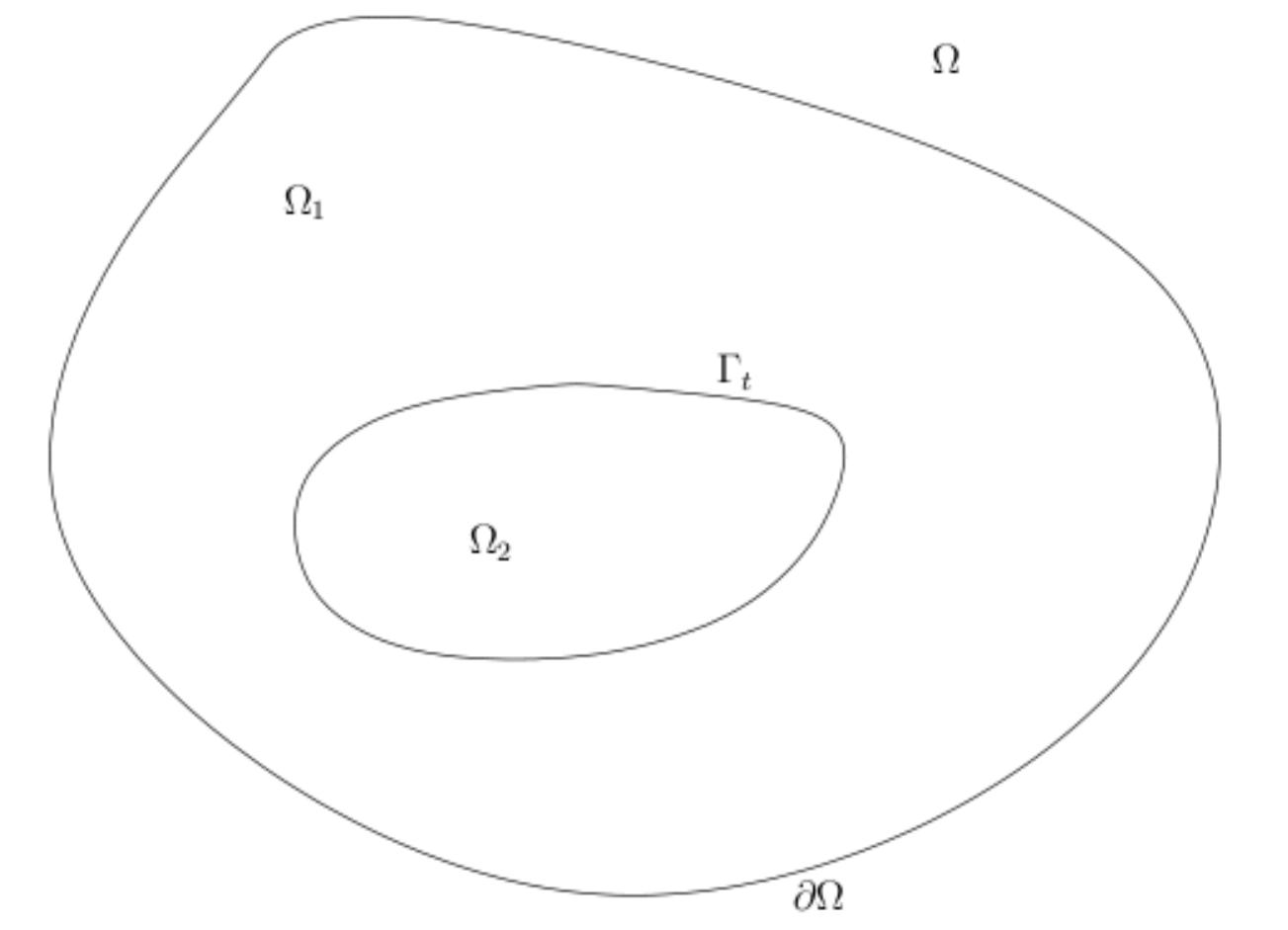}
\caption{Fluid 1 is in region $\Omega_1$ and fluid 2 is region $\Omega_2$.  $\Gamma_t$ represents the mixing interface.}
\label{fluidpic}
\end{figure}

\noindent The overall system includes the Navier-Stokes equation in each $\Omega_i$, where $i=1\hspace{0.05in} \mbox{or} \hspace{0.05in} 2$:

\begin{eqnarray*}\rho_i\left ( \pd{\vec{u}_i}{t} + \left ( \vec{u}_i\cdot\mygrad\right )\vec{u}_i\right ) + \mygrad p_i &=& \mu_i\mygrad^2\vec{u}_i,\\
\mydiv\vec{u}_i &=& 0.\\
\end{eqnarray*} 
Here, $\vec{u}_i$, $p_i$, $\rho_i$, and $\mu_i$ are the velocities, pressures, densities, and viscosities of each fluid, respectively.  Two types of boundary conditions are possible.  Kinematic boundary conditions yield $\vec{u}_i\cdot\vec{n} = V_n$, where $\vec{n}$ is the ``directional" normal of $\Gamma_t$ and $V_n$ is the normal speed of $\Gamma_t$.  Traction free boundary conditions (also known as force balance conditions) satisfy $\left [ \tau_i \right ] \cdot \vec{n} = $ surface force on $\Gamma_t,$ where $\left [\tau_i \right ]$ is the jump of the stress tensor,
$$\tau_i = 2\mu_i\mat{D}_i + p_i\mat{I}, \hspace{0.3in} \mat{D}_i = \dfrac{\mygrad\vec{u}_i + (\mygrad\vec{u}_i)^T}{2}.$$

\begin{rem} Assume $\vec{u}_2|_{\partial\Omega} = 0$.  With the traction free boundary condition, this implies $\Gamma_t \cap \partial\Omega = \emptyset$.
\end{rem}

\noindent In case the surface force is attributed to the surface tension, i.e., $\vec{F}_{\mbox{surface}} = kH\vec{n},$
where $H$ is the mean curvature of $\Gamma_t$ and $k$ is the surface tension parameter, then the whole system possesses the underlying energy law,

$$\pd{}{t}\left ( \sum_{i=1}^2 \int_{\Omega_i} \frac{1}{2} \rho_i |\vec{u}_i|^2 d\vec{x} + k\mbox{Area}(\Gamma_t) \right ) = -\left ( \sum_{i=1}^2 \int_{\Omega_i} \mu_i |\mygrad\vec{u}_i|^2 d\vec{x} \right ).$$

In this paper, we take a different approach, called the diffusive interface method (phase-field method).  The label function, $\phi(\vec{x},t)$, is introduced such that
$$\phi(\vec{x},t) = \left \{ \begin{array}{cc} +1 & \mbox{in fluid 1,}\\-1 & \mbox{in fluid 2.}\\ \end{array} \right .$$
The interaction between the two materials is reflected through the ``mixing" energy, or the Ginzburg-Landau type of energy:
$$\int_{\Omega} W(\phi)  d\vec{x} = \int_{\Omega} \frac{1}{2}|\mygrad\phi|^2 + \frac{1}{4\epsilon^2}\left (\phi^2 - 1\right )^2 d\vec{x},$$
where $\int_{\Omega} \frac{1}{2}|\mygrad\phi|^2d\vec{x}$ represents the tendency of the mixture to be homogeneous.  In other words, it is the ``phillic" interaction.  In addition, $\int_{\Omega} \frac{1}{4\epsilon^2} \left ( \phi^2-1\right )^2 d \vec{x}$ represents the ``phobic" interaction, or the tendency of the mixture to be separated.  The parameter $\epsilon$ reflects the competition between the two opposite tendencies.  

Due to the nature of this empirical energy, it is present in many theories in physics, such as superconductivity \cite{1989GennesP-aa,1950GinzburgV_LandauL-aa} and phase transitions \cite{2002ChenL-aa,2010HyonY_KwakD_LiuC-aa}.  It can be shown that $\phi\rightarrow \pm 1$ almost everywhere as $\epsilon\rightarrow 0$, thus representing the separation of the two immiscible materials.  Moreover, the thermodynamics and patterns of the interface, $\Gamma_t$, can be captured by the level sets of $\phi$, for instance, $\left \{ \vec{x} | \phi(\vec{x},t) = 0 \right \}$.  In addition, the energy $\epsilon \int_{\Omega} W(\phi) \rightarrow c_0 \mbox{Area}(\Gamma_t)$, which is equivalent to the surface tension case in the sharp interface formulation \cite{2004YueP_FengJ_LiuC_ShenJ-aa}.  It is worth pointing out that, while one can view the diffusive interface formula as being the regularizing approximation of the sharp interface formulation, it is more physical to view the sharp interface formulation as being the ideal approximation of the diffusive interface method.  Therefore, it can be understood that the phase field description really captures the microscopic interaction of the mixtures. 

In the examples presented here, we consider problems where we have one or more bubbles of one phase interacting in a region of a fluid of the other phase.  From this system, the continuous energy laws are derived as is the continuous system of PDEs.  We then show that these laws are adequately preserved numerically.  

To solve the system, we use the first-order system least-squares (FOSLS) finite element approach.  This method has been well-tested on many different types of PDEs and, specifically, on several complex fluid systems \cite{1994CaiZ_LazarovR_ManteuffelT_McCormickS-aa,1997CaiZ_ManteuffelT_McCormickS-aa,1998BochevP_CaiZ_ManteuffelT_McCormickS-aa,1999BochevP_ManteuffelT_McCormickS-aa,2007HeysJ_LeeE_ManteuffelT_McCormickS-aa}, including those of magnetohydrodynamics (MHD) \cite{2010AdlerJ_ManteuffelT_McCormickS_RugeJ-aa,2010AdlerJ_ManteuffelT_McCormickS_RugeJ_SandersG-aa}.  We show that use of the FOSLS method allows both an accurate solution of the system and preservation of the energy laws.  Moreover, because we are using this method, we have the advantage of being able to use other numerical techniques, such as nested iteration and adaptive local refinement, which greatly improve the efficiency of solving these problems.  

This paper is structured as follows.  In Section \ref{equations}, we describe the two-phase model and present the equations to be solved.  In section \ref{fosls}, we discuss the FOSLS formulation that is used and show how it presents a well-posed discrete problem.  Section \ref{energetics} introduces the energy laws and discusses what quantities are preserved with the numerical schemes being used.  Numerical results for certain bubble test problems and an MHD simulation are presented in section \ref{results}, where it is shown that this approach is accomplishing the intended goal.  Finally, a discussion of the results and summary is given in section \ref{conclude}.

\section{The Two-Phase Equations}\label{equations}

The dynamics of mixing between two fluids plays an important role in many physical applications.  Here, we use a model that couples the Navier-Stokes equations with a phase-field transport equation.  We use the Allen-Cahn-type equations, which were originally introduced in \cite{1979AllenS_CahnJ-aa} to describe the motion of anti-phase boundaries in crystalline solids.  The Allen-Cahn equations have since been used in many interface problems in fluid dynamics through a phase-field approach \cite{2002ChenL-aa,2010HyonY_KwakD_LiuC-aa,2004YueP_FengJ_LiuC_ShenJ-aa,2003LiuC_ShenJ-aa,2010ShenJ_YangX-aa}.  They are

\begin{eqnarray}
\label{ns}\pd{\vec{u}}{t} + \left (\vec{u}\cdot\mygrad\right ) \vec{u} + \mygrad p + \lambda\mydiv\left ( \mygrad\phi\otimes\mygrad\phi \right ) - \mu \mygrad^2\vec{u} & = & 0,\\
\label{divu}\mydiv\vec{u} &=&0,\\
\label{phase}\pd{\phi}{t} + \vec{u}\cdot\mygrad\phi - \gamma\left ( \mygrad^2\phi - \frac{1}{\epsilon^2}\phi\left ( \phi^2-1\right ) \right ) & = & 0.\\ \nonumber
\end{eqnarray}
Here, $\vec{u}$ is the fluid velocity, $p$ is the fluid pressure, and $\phi$ is the phase-field function.  The physical parameters are the mixing energy parameter, $\lambda$, the viscosity of the two fluids, $\mu$, the elastic relaxation time of the system,  $\gamma$, and the mixing layer width, $\epsilon$.  We assume that both fluids have the same viscosity and mass density (which is assumed to be 1).  It should be noted that as $\epsilon \rightarrow 0$, the parameter $\lambda$ is associated with the surface tension at the boundary.  

With the ``mixing" energy, $W(\phi)$, described above, and following the general approach of Onsager \cite{1931OnsagerL-aa,1931OnsagerL-ab,1953OnsagerL_MachlupS-aa,1953MachlupS_OnsagerL-aa}, we can prescribe the whole system as a dissipative system:
$$\pd{}{t} \left ( \int_{\Omega} \frac{1}{2} \rho(\phi) |\vec{u}|^2 + \lambda W(\phi) d\vec{x} \right ) = - \int_{\Omega} \mu(\phi)|\mygrad\vec{u}|^2 + \frac{\lambda}{\gamma} |\dot{\phi}|^2 d\vec{x},$$
where $\int_{\Omega} |\dot{\phi}|^2 d\vec{x}$ represents the microscopic internal damping, with the kinematic assumption of $\dot{\phi} = \pd{\phi}{t} + \vec{u}\cdot\mygrad\phi$.  We can then employ the least action principle (LAP) for the Hamiltonian part of the system and the maximum dissipation principle (MDP) for the dissipative part of the system \cite{1931OnsagerL-aa,1931OnsagerL-ab,1953OnsagerL_MachlupS-aa,1953MachlupS_OnsagerL-aa,1963GelfandI_FominS-aa}.  In the case that $\mu_1 = \mu_2 = \mu$ and $\rho_i = 1$, we obtain equations (\ref{ns})-(\ref{phase}).  In case $\gamma=0$, i.e., we neglect the internal damping, we can show that (\ref{ns})-(\ref{phase}) approaches the sharp interface formulation \cite{2010ShenJ_YangX-aa,1994ChenX-aa}.  In particular,
$$\lambda\mydiv\left ( \mygrad \phi \otimes \mygrad \phi \right ) = \lambda \left (\mygrad^2 \phi \mygrad \phi + \mygrad \left ( \frac{|\mygrad\phi|^2}{2} \right )\right ),$$
where $\lambda \left (\mygrad^2 \phi \mygrad \phi \right )$ reflects the surface tension, $kH\vec{n}$, modulating the pressure contributions.  

\begin{rem} Later, we indicate the structural similarities between the diffusive interface formulation and that of resistive MHD systems.  In this setting, the terms $\lambda\mydiv\left ( \mygrad \phi \otimes \mygrad \phi \right ) - \mygrad \left ( \frac{|\mygrad\phi|^2}{2} \right )$ are equivalent to the Maxwell stress, and $\lambda \mygrad^2 \phi \mygrad \phi $ is equivalent to the Lorentz force.  
\end{rem}

Next, we discuss the numerical methods that are used to solve this problem.  
\section{FOSLS Formulation}\label{fosls}

To solve the nonlinear system of equations, which for discussion is denoted by $\mathcal{L}(u)=f$, a
first-order system least-squares discretization (FOSLS) is used \cite{1994CaiZ_LazarovR_ManteuffelT_McCormickS-aa,1997CaiZ_ManteuffelT_McCormickS-aa,1994PehlivanovA_CareyG_LazarovR-aa}.  First, consider a
linear first-order system, denoted by $Lu=f$. Then, the linear problem is recast as the minimization of a functional constructed by taking
the $L^2$ norm of the residual of each equation. This is written as

\begin{equation}\label{argmin}u_* = \mbox{arg}\min_{u \in \mathcal{V}}  G(u;f) := \mbox{arg}\min_{u \in \mathcal{V}}  ||Lu-f||_0^2,\end{equation}
where $u_*$ is the solution in an appropriate Hilbert space $\mathcal{V}$.\\

\noindent Thus, $u_*$ satisfies $[G'(u_*)]v = 0$, which is the Fr\'{e}chet
derivative of G at $u_*$ applied to $v\in \mathcal{V}$.  This results in the weak form of the problem:\\

find $u_* \in \mathcal{V}$ such that
\begin{equation}\label{weakform}\langle Lu_*,Lv\rangle = \langle f,Lv\rangle \hspace{0.2in}\forall v\in \mathcal{V},\end{equation}

\noindent where $\langle\cdot,\cdot\rangle$ is the usual $L^2$ inner product.  Desirable properties of the bilinear form, $\langle Lu,Lv\rangle$, are
\begin{eqnarray}
\label{cont}\mbox{continuity}&\hspace{0.3in}\langle Lu,Lv\rangle \hspace{0.1in}\leq \hspace{0.1in}c_2||u||_{\mathcal{V}}||v||_{\mathcal{V}}&\hspace{0.2in}\forall\hspace{0.03in} u,v \in \mathcal{V},\\
\label{coerc}\mbox{coercivity}&\langle Lu,Lu\rangle \hspace{0.1in}\geq\hspace{0.1in} c_1||u||^2_{\mathcal{V}}&\hspace{0.2in}\forall\hspace{0.03in} u \in \mathcal{V}. \\ \nonumber
\end{eqnarray}

\noindent These properties imply the existence of a unique solution, $u_*\in \mathcal{V}$, for the weak problem (\ref{weakform}).\\

Next, we approximate $u_*$ by restricting (\ref{argmin}) to a finite-dimensional subspace, $\mathcal{V}^h \subseteq \mathcal{V}$, which leads to (\ref{weakform}) restricted to $\mathcal{V}^h$.  Since $\mathcal{V}^h$ is a subspace of $\mathcal{V}$, the discrete problem is also well-posed.  Choosing an appropriate basis, $\mathcal{V}^h = span\{\Phi_j\}$, yields an algebraic system of equations involving the matrix $A$ with elements
$$(A)_{ij} = \langle L\Phi_j,L\Phi_i\rangle.$$
In general, if $\mathcal{V}$ is a subspace of $H^1$ (product space), and continuity and coercivity hold we say that the FOSLS functional is $H^1$ equivalent.  In this case, the linear system is amenable to an iterative solution by multilevel techniques.  In particular, AMG has been shown to work well on a wide range of such problems \cite{1994CaiZ_LazarovR_ManteuffelT_McCormickS-aa,1997CaiZ_ManteuffelT_McCormickS-aa,1998BochevP_CaiZ_ManteuffelT_McCormickS-aa,1999BochevP_ManteuffelT_McCormickS-aa,1986BrandtA-aa,2000BriggsW_HensonV_McCormickS-aa,2001CoddA-aa,2001TrottenbergU_OosterleeC_SchullerA-aa,2004RoehrleO-aa,1987RugeJ_StubenK-aa,2004WestphalC-aa}.

In addition, note that the functional yields an a posteriori error measure.  If $u^h \in \mathcal{V}^h$, then
\begin{equation}\label{apost}G(u^h;f) = ||Lu^h-f||_0^2 = ||Lu^h-Lu_*||_0^2 = ||Le^h||_0^2 \approx c||e^h||^2_{\mathcal{V}}.\end{equation}
The last relation comes from the continuity and coercivity bounds in (\ref{cont}) and (\ref{coerc}).  Thus, the functional value is equivalent to the error measured in the Hilbert space norm.  In general, the constant $c$ in equation (\ref{apost}) depends on the continuity and coercivity constants, $c_2$ and $c_1$.  These constants, of course, depend on properties of the PDE as well as on boundary conditions and the computational domain.  For this paper, we study simple domains and boundary conditions for which these constants are not large.  In addition, our numerical results show that the problem parameters do not adversely affect our accuracy either.  Addressing these issues analytically would be a topic of another paper and has been studied in other problems such as in \cite{1998BochevP_CaiZ_ManteuffelT_McCormickS-aa,1999BochevP_ManteuffelT_McCormickS-aa,2007HeysJ_LeeE_ManteuffelT_McCormickS-aa}.  Locally, on any subset of $\Omega$, such as an element, the functional yields an estimate of the error.  This property of FOSLS helps make it possible to efficiently solve complex systems.  At each step in the solution algorithm, a local  measure of the functional is available.  Since this is the norm we are minimizing, this allows judgements to be made based on estimates of the increase of accuracy that results from an increase in computational cost.  As a result, methods such as nested iteration and adaptive local refinement can be used to improve the efficiency of solving such systems.  This has been applied to various complex fluids problems, including magnetohydrodynamics \cite{2010AdlerJ_ManteuffelT_McCormickS_RugeJ-aa,2010AdlerJ_ManteuffelT_McCormickS_RugeJ_SandersG-aa,2010AdlerJ_ManteuffelT_McCormickS_NoltingJ_RugeJ_TangL-aa,2010AdlerJ_ManteuffelT_McCormickS_NoltingJ_RugeJ_TangL-ab}.  A brief description of the algorithms used is described in section \ref{results}.

\subsection{Linearization (Newton-FOSLS)}
Since the two-phase system is nonlinear, we first linearize it before we put it into a FOSLS weak form.  The system becomes
\begin{equation}\mathcal{L}(u_0+\hat{u}) \approx \mathcal{L}(u_0) + \left [\mathcal{L}'(u_0)\right ]\hat{u},\end{equation}
where $u_0$ is the current approximation and $\left [\mathcal{L}'(u_0)\right ]$ is the Fr\'{e}chet
derivative evaluated at $u_0$. The functional then becomes
\begin{equation}\label{linfunc}G(\hat{u}, f) = ||\left [\mathcal{L}'(u_0)\right ]\hat{u}- (f-\mathcal{L}(u_0))||_0^2.\end{equation}
Minimization of the linearized functional yields $\hat{u}_*$ that satisfies the weak form:\\

find $\hat{u}_* \in \mathcal{V}$ such that,
\begin{equation}\langle \left [\mathcal{L}'(u_0)\right ]\hat{u}_*,\left [\mathcal{L}'(u_0)\right ]\hat{v}\rangle = \langle(f-\mathcal{L}(u_0)),\left [\mathcal{L}'(u_0)\right ]\hat{v}\rangle\hspace{0.2in} \forall\hspace{0.05in}\hat{v} \in \mathcal{V}.\end{equation}

\noindent Once $\hat{u}$ is found, it is added to the previous iterate to get
the next approximation,
\begin{equation}u_1 = u_0 + \hat{u}.\end{equation}

\noindent Thus, in a nonlinear setting, the FOSLS approach can be applied
and, if the linearized functional retains continuity and coercivity, (\ref{cont}) and (\ref{coerc}),
it also retains the desirable
properties as described for linear systems.  In addition, the nonlinear functional, that is, the $L^2$ norm of the residual of the nonlinear system, can be computed as well as the linear functional or the $L^2$ norm of the linearized system.  This allows the two functional values to be compared after each linearization and helps determine if the Newton iterations are converging as expected.

A similar approach would be to create the FOSLS functional of the nonlinear problem and linearize this functional instead.  This FOSLS-Newton approach generally involves more terms than the Newton-FOSLS method described here.  As the approximation approaches the solution, these additional terms are higher-order and tend to zero faster than the overall functional and the two methods tend to be the same.  Thus, when nested iteration is used, as it is here, there tends to be very little difference in the convergence behavior of these two approaches.  FOSLS-Newton may be more robust in some especially challenging situations, but the Newton-FOSLS approach is simpler and has been successful in a number of applications \cite{2001CoddA-aa,2003CoddA_ManteuffelT_McCormickS-aa}, so we confine our presentation to it in this paper.

\subsection{Velocity-Grad Formulation}

To reformulate (\ref{ns})-(\ref{phase}) as a first-order system, we introduce new variables that represent the gradient of the velocity vector, $\vec{u}$, and the phase-field function, $\phi$:
\begin{equation}\mat{V} = \mygrad\vec{u} = \left ( \begin{array}{cc} v_{11}&v_{21}\\v_{21}&v_{22} \\\end{array} \right ) = \left ( \begin{array}{cc}\partial_x u_1&\partial_x u_2\\\partial_y u_1&\partial_y u_2\\\end{array} \right ),\end{equation}
\begin{equation}\label{Bgradphi}\vec{B} = \mygrad\phi = \left ( \begin{array}{c}\partial_x \phi\\\partial_y \phi\\\end{array} \right ).\end{equation}

This formulation has been shown to work well for the Navier-Stokes equations alone when velocity boundary conditions are given \cite{1998BochevP_CaiZ_ManteuffelT_McCormickS-aa,1999BochevP_ManteuffelT_McCormickS-aa}.  To make the system $H^1$ equivalent, a few auxiliary equations are added to the system.  These are consistent with the original equations and make the first-order system well-posed.  The resulting system is as follows:

\begin{eqnarray}
\label{gradV} \mat{V} - \mygrad \vec{u} & = & \mat{0},\\
\label{curlV} \curl\mat{V} & = & \vec{0},\\
\label{divu1}\mydiv\vec{u} &=&0,\\
\label{trace}\mygrad \left ( tr\mat{V} \right ) &=& \vec{0},\\
\label{nsfosls}\pd{\vec{u}}{t} + \mygrad p + \lambda\vec{B}\left ( \mydiv\vec{B} \right ) - \mu \mydiv\mat{V} & = & \vec{0},\\
\label{gradphi}\vec{B} - \mygrad\phi &=&\vec{0},\\
\label{curlB} \curl\vec{B} &=&0,\\
\label{phase1}\pd{\phi}{t} + \vec{u}\cdot\vec{B} - \gamma\left ( \mydiv\vec{B} - \frac{1}{\epsilon^2}\phi\left ( \phi^2-1\right ) \right ) & = & 0.\\ \nonumber
\end{eqnarray}

In practice, equation (\ref{trace}) is achieved by eliminating one of the diagonal components of $\mat{V}$ in terms of the other diagonal components.  This comes from the divergence-free constraint on the velocity.  Thus, in 2D, there are 13 equations and 9 unknowns and, in 3D, there are 29 equations and 16 unknowns.  

\subsection{Equivalence of the FOSLS Functional}
As described above, if continuity, (\ref{cont}), and coercivity, (\ref{coerc}), of the FOSLS functional holds, we say the functional is $H^1$ equivalent.  This can be shown for the linearized system at each time step.  We assume that an implicit backward difference time-stepping scheme is used and consider this semi-discrete system first.  This says that, for each Newton step and time step, the FOSLS discretization gives a well-posed numerical system to solve.  We state the main theorem here and give the proof.  However, many of the details are omitted to save space and since they have been shown in previous works.     

\begin{thm}[cf. \cite{1994CaiZ_LazarovR_ManteuffelT_McCormickS-aa,1997CaiZ_ManteuffelT_McCormickS-aa,1998BochevP_CaiZ_ManteuffelT_McCormickS-aa,1999BochevP_ManteuffelT_McCormickS-aa,2009AdlerJ-aa}]\label{mainthm}
Let $\mathcal{L}$ be the linearized first order two-phase flow operator as in equations (\ref{gradV})-(\ref{phase1}) and let $G$ be the FOSLS functional as defined in equation (\ref{linfunc}) of this linearized system.  Let the domain, $\Omega$, be a bounded convex polyhedron with connected boundary or a simply connected region with a $C^{1,1}$ boundary.  Denote $\mathcal{U} = (\vec{u}, \mat{V}, p, \phi, \vec{B})^T$ and let
\noindent$$\mathcal{U} \in \mathcal{V} := [{\bf H^1}(\Omega)]^2 \times  [{\bf H^1}(\Omega)]^4 \times [L_0^2(\Omega) \cap H^1(\Omega)] \times [H^1_0(\Omega)] \times [{\bf H^1}(\Omega)].$$

\noindent Assume $\vec{u}^n$, $\mat{V}^n$, $\vec{B}^n$, and $\phi^n$ are given members of our finite-element space and, therefore,\\
$||\vec{u}^n||_{\infty}$, $||\mat{V}^n||_{\infty}$, $||\vec{B}^n||_{\infty}$, $||\mydiv \vec{B}^n||_{\infty}$, and $||\phi^n||_{\infty} < \infty$.  Then there exists positive constants, $c_1$ and $c_2$, that yield the following bounds:

\begin{itemize}
\item[a.] \be\label{apart} G  \leq c_2 ( ||\vec{u}||_1^2 + ||\mat{V}||_1^2 + ||p||_1^2 + ||\vec{B}||_1^2 + ||\phi||_1^2),  \ee
\item[b.] \be\label{bpart} c_1 ( ||\vec{u}||_1^2 + ||\mat{V}||_1^2 + ||p||_1^2 + ||\vec{B}||_1^2 + ||\phi||_1^2) \leq  G. \ee
\end{itemize}
\noindent These bounds mean that the functional, G, is equivalent to the $H^1$ norm of the error.  
\end{thm}

\noindent The spaces above are defined as follows:
\begin{eqnarray*}L_0^2(\Omega) &=& \{ p \in L^2(\Omega) : \int_{\Omega} p dx = 0 \},\\
H^1(\Omega) &=& \{ p\in L^2(\Omega): \mygrad p \in (L^2(\Omega))^3 \},\\
 {\bf H^1}(\Omega)^d \hspace{0.05in} &=&\mbox{the d-dimensional version of} \hspace{0.05in} H^1(\Omega),\\
H^1_0(\Omega) &=& \{ p\in L^2(\Omega): \mygrad p \in (L^2(\Omega))^3 \hspace{0.05in} \mbox{and} \hspace{0.05in} p = 0 \hspace{0.05in} \mbox{on} \hspace{0.05in} \partial\Omega \}. \\
\end{eqnarray*}

\begin{proof}[Proof of \ref{mainthm}]
Proving continuity, (\ref{apart}), of the linearized first-order system involves several instances of the triangle inequality and is, therefore, easy to show. 

Coercivity, (\ref{bpart}), is accomplished by treating the two coupled block systems, Navier-Stokes and the phase field equation, separately.  A linearized system of (\ref{gradV})-(\ref{phase1}) can be written in block form as follows:

$$\mathcal{A}\mathcal{U} = \left ( \begin{array}{cc} A_{NS}&A_{NP}\\A_{PN}&A_{phase}\\ \end{array} \right ) \left ( \begin{array}{c}\vec{u}\\\mat{V}\\p\\\phi\\\vec{B}\\\end{array} \right ) .$$
Here, $A_{NS}$ represents the linearized Navier-Stokes equations in first-order form and without the coupling term, (\ref{gradV})-(\ref{nsfosls}).  Similarly, $A_{phase}$ represents the linearized phase function equations without the velocity field coupling, (\ref{gradphi})-(\ref{phase1}).  The coupling terms are as follows:

$$A_{NP}\vec{B} = \lambda\vec{B}^n \mydiv\vec{B} + \lambda(\mydiv\vec{B}^n)\vec{B},$$
$$A_{PN}\vec{u} = \vec{B}^n\cdot\vec{u}.$$
The superscript $n$ represents that the solution is from a previous Newton step.  Since we are working on discrete finite element spaces, we can assume the $L^{\infty}$ norms of all the previous iterations are bounded. 

In \cite{1998BochevP_CaiZ_ManteuffelT_McCormickS-aa,1999BochevP_ManteuffelT_McCormickS-aa}, coercivity has been shown for the $A_{NS}$ block given here.  For the phase field system, the coercivity bound can be shown in a similar fashion.  In fact, ignoring the nonlinear terms, the system resembles the FOSLS formulation used on the heat equation \cite{1994CaiZ_LazarovR_ManteuffelT_McCormickS-aa,1997CaiZ_ManteuffelT_McCormickS-aa}:

$$\pd{\phi}{t}-\mydiv\vec{B} = f,$$
$$\vec{B}-\mygrad\phi = 0,$$
$$\curl\vec{B} = 0,$$
$$\phi = 0 \hspace{0.1in} \mbox{on} \hspace{0.1in} \partial\Omega,$$
$$\vec{n}\times\vec{B} = 0 \hspace{0.1in} \mbox{on} \hspace{0.1in} \partial\Omega.$$
Forming the functional of the linearized system for just the phase field block, we write

$$G_{phase} = ||\alpha\phi + \vb\cdot\vec{B} - \gamma \mydiv\vec{B}||_0^2 + ||\curl\vec{B}||_0^2 + ||\vec{B}-\mygrad\phi||_0^2.$$
Again, it is assumed that an implicit time-stepping scheme is used so $\alpha$ depends on the time step.  With part of the nonlinear coupling included in this block, $\vb$ depends on the parameters of the problem and the previous solutions.  This in fact makes the problem resemble an advection-diffusion type system.  Since this functional only adds zeroth-order terms (i.e. no derivatives) to this system, a standard compactness argument, \cite{1979GiraultV_RaviartP-aa}, can be used to show that there exists a constant, $C>0$, such that
$$C\left (||\vec{B}||_1^2 + ||\phi||_1^2\right ) \leq G_{phase}.$$
It should be noted that this requires $c||\vec{B}||_1^2 \leq ||\mydiv\vec{B}||_0^2 + ||\curl\vec{B}||_0^2$, for some constant $c>0$.  However, this is true due to the assumptions on the boundary and the boundary condition, $\vec{n}\times\vec{B}=0$ \cite{1979GiraultV_RaviartP-aa}.  Thus, the uncoupled system is $H^1$ equivalent:  there exists positive constants, $\hat{c}_L$ and $\hat{c}_U$, such that

$$\hat{c}_L||\mathcal{U}||_1^2 \leq ||\left ( \begin{array}{cc}A_{NS}&0\\0&A_{phase}\\\end{array} \right )\mathcal{U}||_0^2 \leq \hat{c}_U||\mathcal{U}||_1^2.$$

Next, we consider adding the off-diagonal coupling blocks of the system.  In \cite{2009AdlerJ-aa}, the following lemma was proved:
\begin{lem}\label{upptriag}
Let $\mathcal{L}$ be a $2\times2$ upper triangular block matrix such that
$$\mathcal{L}\mathcal{U} = \left ( \begin{array}{cc} A_1 &T\\
0&A_2\\\end{array} \right ) \left ( \begin{array}{c}
\vec{u}_1\\\vec{u}_{2}\\\end{array} \right ) = \left ( \begin{array}{c}\vec{f}\\\vec{g}\\\end{array} \right ). $$

\noindent Let $A_1$ and $A_2$ be invertible and let there exist a positive constant, $C$, such that
$$||T\vec{u}_2||_0^2 \leq C||A_2\vec{u}_2||_0^2.$$

\noindent Then
$$||\left ( \begin{array}{cc} A_1 &0\\
0&A_2\\\end{array} \right ) \left ( \begin{array}{c}
\vec{u}_1\\\vec{u}_{2}\\\end{array} \right )||_0^2 \leq (1 + C) ||\left ( \begin{array}{cc} A_1 &T\\
0&A_2\\\end{array} \right ) \left ( \begin{array}{c}
\vec{u}_1\\\vec{u}_{2}\\\end{array} \right )||_0^2. $$
\end{lem}
\vspace{0.2in}

Looking at the upper triangular portion of the two-phase flow system, we see that
$$||A_{NP}\vec{B}||_0^2 =  ||\lambda\vec{B}^n \mydiv\vec{B} + \lambda(\mydiv\vec{B}^n)\vec{B}||_0^2 \leq C_1||\vec{B}||_0^2 + C_2||\mydiv\vec{B}||_0^2,$$
for some positive constants $C_1$ and $C_2$.  Adding a curl term at the end and using a Poincar\'{e} inequality, it is easily shown from the coercivity of the phase-field block that 
$$||A_{NP}\vec{B}||_0^2 \leq C_3||A_{phase}\vec{B}||_0^2,$$
for some constant, $C_3>0$.  Thus, from Lemma \ref{upptriag}, we know that

$$\frac{\hat{c}_L}{1+C_3}||\mathcal{U}||_1^2 \leq ||\left ( \begin{array}{cc}A_{NS}&A_{NP}\\0&A_{phase}\\\end{array} \right )\mathcal{U}||_0^2$$

Finally, to show equivalence of the full linearized system, it should be noted that the lower left block, $A_{PN}$, involves a term with no derivatives and another standard compactness argument can be used.  Here, again, we must assume that the previous iterations of $\vec{u}$ and $\vec{B}$ are bounded.  Thus, there exists positive constants $c_1$ and $c_2$ such that
$$c_1||\mathcal{U}||_1^2 \leq ||\left ( \begin{array}{cc}A_{NS}&A_{NP}\\A_{PN}&A_{phase}\\\end{array} \right )\mathcal{U}||_0^2 \leq c_2||\mathcal{U}||_1^2.$$
\end{proof}

Therefore, for each time step and linearization, there exists a weak solution to our linear system, and our multilevel solvers will converge to that solution.  The constants $c_1$ and $c_2$ depend on physical parameters such as $\gamma$, $\mu$, and $\epsilon$, as well as the previous Newton step iterations and time step approximations.  However, they do not depend on the number of degrees of freedom or the size of the finite element grids being used.   In the next section, we discuss the energetics of complex fluid flow.

\section{Energetics}\label{energetics}

While most physical systems can be described by a set of PDEs, usually an underlying energy law also describes the system.  In many cases, the PDE is actually derived from the energy laws via a calculus of variation (see e.g. \cite{1963GelfandI_FominS-aa}).  The main idea is that the change in the total energy of a system over time must equal the total dissipation of the system.  If the system is conservative, or there is no dissipation, then the total change is zero and it is a Hamiltonian system.  Thus, the interaction or coupling
between different scales or phases plays a crucial role in understanding complex fluids.  Any set of equations that describe the system should then be a result of the energy laws. 
The energetic variational approach (EVA) takes the energy laws of a complex fluid model and,
using the least action principle (LAP) and the maximum/minimum dissipation
principle (MDP) \cite{1931OnsagerL-aa,1931OnsagerL-ab,1953OnsagerL_MachlupS-aa,1953MachlupS_OnsagerL-aa}, yields a weak form of the system, which we approximately preserve in our numerical simulations.  The energy variation is based on the following energy dissipation law for the whole
coupled system:

\be\label{evalaw} \pd{E^{total}}{t} = -\mathcal{D},\ee
where $E^{total}$ is the total energy of the system and $\mathcal{D}$ is the dissipation. The
LAP, which is also referred to as the Hamiltonian principle, or principle of virtual work,
gives us the Hamiltonian (reversible) part of the system related to the conservative
force. At the same time, the MDP gives the dissipative (irreversible) part of the system
related to the dissipative force.

In \cite{2010HyonY_KwakD_LiuC-aa,2004YueP_FengJ_LiuC_ShenJ-aa,1979AllenS_CahnJ-aa,2003LiuC_ShenJ-aa,2010ShenJ_YangX-aa} and others, an energy law is developed for the Allen-Cahn phase-field model described above.  The total energy is a combination of the kinetic energy (driven by the fluid) and the internal energy (driven by the interface):

$$E^{total} = \int_{\Omega} \left \{ \frac{1}{2}|\vec{u}|^2 + \lambda W(\phi) \right \} d\vec{x},$$
where $W(\phi)$ is the Ginzburg Landau mixing energy described above \cite{2010HyonY_KwakD_LiuC-aa}, thus representing a competition between two fluids with their hydrophilic and hydrophobic properties.  The dissipation is a result of a diffusive term from the viscosity of the fluid and from the diffusion of the interface itself:
$$\mathcal{D} =  \int_{\Omega} \left \{ \mu|\mygrad\vec{u}|^2 + \frac{\lambda}{\gamma} | \dot{\phi} |^2 \right \} d\vec{x} = \int_{\Omega} \left \{ \mu|\mygrad\vec{u}|^2 + \lambda\gamma \left | \mygrad^2\phi - \frac{1}{\epsilon^2}\left ( \phi^2 - 1 \right ) \phi \right |^2 \right \} d\vec{x}.$$

\noindent Thus, any system describing this behavior, such as equations (\ref{ns})-(\ref{phase}) or the FOSLS formulation (\ref{gradV})-(\ref{phase1}), should globally produce results that approximately satisfy the following energy law:

\begin{equation} \dfrac{d}{dt} \int_{\Omega}   \left \{ \frac{1}{2}|\vec{u}|^2 + \lambda W(\phi) \right \} d\vec{x} = - \int_{\Omega} \left \{ \mu|\mygrad\vec{u}|^2 + \lambda\gamma \left | \mygrad^2\phi - \frac{1}{\epsilon^2}\left ( \phi^2 - 1 \right ) \phi \right |^2 \right \} d\vec{x}. \end{equation} 

The overall energy law reflects the multi-scale, multi-physics nature of the system.  As mentioned above, the variable $\phi$ represents the microscopic description of the mixture.  The kinematic assumption, $\dot{\phi} = \pd{\phi}{t} + \vec{u}\cdot\mygrad\phi$, stands for the influence of macroscopic dynamics on the microscopic scale.  The form of the internal dissipation is the consequence that we are looking at the long (or macroscopic) time dynamics.  This brings us to the equation for $\phi$, 
$$\pd{\phi}{t} + \vec{u}\cdot\mygrad\phi = \gamma \left ( \mygrad^2 \phi - \frac{1}{\epsilon^2} \left ( \phi^2 -1 \right ) \phi \right ). $$
This ``gradient flow" dynamics is another formulation of the near equilibrium, linear response theory \cite{1953OnsagerL_MachlupS-aa,1953MachlupS_OnsagerL-aa}.  

As for the macroscopic force balance, the LAP gives the following Hamiltonian system, with $\gamma = 0$:
\begin{eqnarray*} \pd{\vec{u}}{t} + \left (\vec{u}\cdot\mygrad\right ) \vec{u} + \mygrad p &=& -\lambda\mydiv\left ( \mygrad\phi\otimes\mygrad\phi\right ),\\
\mydiv\vec{u} &=& 0,\\
\pd{\phi}{t} + \vec{u}\cdot\mygrad\phi &=& 0.\\
\end{eqnarray*}
This describes the transient dynamics of the system.  

On the other hand, the MDP yields the dissipation of system, which stands for the long macroscopic time dynamics:
\begin{eqnarray*} -\mu\mygrad^2\vec{u} + \mygrad p &=& -\frac{\lambda}{\gamma} \dot{\phi}\mygrad\phi,\\
\mydiv\vec{u} &=& 0,\\
\pd{\phi}{t} + \vec{u}\cdot\mygrad\phi &=& \gamma \left ( \mygrad^2 \phi + \frac{1}{\epsilon^2} \left (\phi^2 - 1 \right )\phi \right ).\\
\end{eqnarray*}

\begin{rem}  $\frac{\lambda}{\gamma}\dot{\phi}\mygrad\phi = \lambda \left ( \mygrad^2 \phi + \frac{1}{\epsilon^2}\left (\phi^2 -1 \right )\phi \right )\mygrad\phi = \lambda\mygrad^2\phi\mygrad\phi + \lambda\mygrad\left (\frac{1}{4\epsilon^2}\left (\phi^2-1\right )^2 \right )$.  This shows the consistency between the LAP and the MDP \cite{2010HyonY_KwakD_LiuC-aa}.
\end{rem}

\noindent System (\ref{ns})-(\ref{phase}) is really the hybrid of these two systems.

\begin{rem} It is worth pointing out that due to the generality of the formulation, analogies of (\ref{ns})-(\ref{phase}) can be found in many different physics.  In particular, without the $\int_{\Omega}\frac{1}{4\epsilon^2}\left (\phi^2-1\right)^2 d\vec{x}$ term, the system is equivalent to the resistive MHD system.  In fact, it is this understanding that allows us to arrive at a similar weak form for the FOSLS algorithms as in \cite{2010AdlerJ_ManteuffelT_McCormickS_RugeJ-aa}.  The choice of using the gradient of $\phi$ as an auxiliary variable in (\ref{Bgradphi}) was made for this reason.  \end{rem}

Finally, we point out that the advantage of the above energetic variation approach enables us to derive the thermodynamic-consistent models involving different physics at different scales.  Restricting to a finite-dimensional space, we show numerically that the FOSLS discretization method is capable of globally preserving this energy law.

\section{Numerical results}\label{results}

In \cite{2010AdlerJ_ManteuffelT_McCormickS_RugeJ-aa,2010AdlerJ_ManteuffelT_McCormickS_RugeJ_SandersG-aa}, an algorithm is devised to solve a system of nonlinear equations, $\mathcal{L}(u) = f$.  Starting on a coarse grid with an initial guess there, the system is linearized and the linearized FOSLS functional is then minimized on a finite element space.  At this point, several algebraic multigrid (AMG) V-cycles are performed until there is little to gain in accuracy-per-computational-cost.  The system is then relinearized and the minimum of the new linearized FOSLS functional is searched for in the same manner. After each set of linear solves, the relative difference between the computed linearized functional and the functional of the nonlinear system is checked.  If they are close and near the minimum of the linearized functional, then it is concluded that we are close enough to the minimum of the functional of the nonlinear operator and, hence, we have a good approximation to the solution on the given grid.  Next, the approximation is interpolated to a finer grid and the problem is solved on that grid.  At this stage, the grid can be refined uniformly or locally.  This process is repeated to yet finer grids until an acceptable error has been reached, or until we have run out of computational resources, such as memory.  If, as in the case of the 2-phase flow equations, it is a time-dependent problem, the whole process is performed at each time step.  This algorithm is summarized in the following pseudocode.



\begin{algorithm}[H] 
\SetAlgoLined 
\For{$t=1$ \KwTo max time step}{
Go to coarsest grid\\
\While{fine grid resolution is not reached}{
\For{$n=1$ \KwTo max Newton step}{
Linearize first-order system\\
Minimize FOSLS functional\\
\While{functional of current approximation is large}{
Solve $Ax = b$ with AMG}
Compute relative difference between linearized and nonlinear operator\\
\If{small}{Exit Newton loop}
}
Refine grid
}
Update time step
}
\caption{Nested Iteration Newton-FOSLS AMG}
\end{algorithm}

The main goal behind this algorithm is to reduce the amount of work needed to solve the system of equations.  Relatively cheap work is done on the coarse computational grid to get better starting guesses for the solution of the system on the finer finite element grids.  Then, less iterations are needed at the resolution that would require the most work.  The algorithm described above also takes into account the fact that resolving too much on coarse grids is unnecessary.  Using the FOSLS functionals as estimates of the errors, stopping criteria is implemented for each stage of the algorithm, optimizing the amount of work per computational cost.

In addition to the nested iteration algorithm, adaptive local refinement is used.  Again, using the FOSLS a posteriori error estimates, an efficiency-based adaptive local refinement scheme, known as ACE \cite{2010AdlerJ_ManteuffelT_McCormickS_NoltingJ_RugeJ_TangL-aa,2008De-SterckH_ManteuffelT_McCormickS_NoltingJ_RugeJ_TangL-aa} can be used fairly easily.  The idea behind this scheme is to predict the amount of work and the error reduction obtained from refining only certain portions of the computational domain.  Then, the optimal grid is chosen.   In simulations of magnetohydrodynamics, which have similar structure to the complex fluid problems described here, this scheme has been shown to reduce the amount of work by a factor of ten in the FOSLS framework \cite{2010AdlerJ_ManteuffelT_McCormickS_NoltingJ_RugeJ_TangL-aa,2010AdlerJ_ManteuffelT_McCormickS_NoltingJ_RugeJ_TangL-ab}.

For the following test problems, we use the FOSLS finite element method along with nested iteration, AMG, and adaptive local refinement.  Here we show that, using these methods, we are able to solve the complex fluid problems well, while still adequately preserving the energetics of the system.  

\subsection{Two-Phase Flow}
\subsubsection{Coalescence}

In this test problem, we start with two bubbles of a certain phase immersed in a fluid of a different phase on the unit square, $\Omega = [0,1]\times[0,1]$.  The initial interface between the two fluids is given such that the two bubbles are osculating or ``kissing", and then the system is allowed to evolve.  As a result, the two bubbles begin to coalesce into one and, eventually, due to the dissipation in the system, the newly formed bubble slowly shrinks and is absorbed by the outside fluid.  The initial conditions for this system are similar to that in \cite{2010HyonY_KwakD_LiuC-aa} and are as follows:

\begin{eqnarray*}
\vec{u} &=&\left ( \begin{array}{c}0\\0\\\end{array} \right ),\\
\phi &=&\tanh \left(\dfrac{d_1(x,y)}{2\eta}\right ) + \tanh \left(\dfrac{d_2(x,y)}{2\eta}\right ) - 1,\\ 
\end{eqnarray*}

\noindent where 
$$d_1(x,y) = \sqrt{(x-0.38)^2+(y-0.5)^2}-0.11,$$
and 
$$d_2(x,y) = \sqrt{(x-0.62)^2+(y-0.5)^2}-0.11.$$
The parameters used are $\eta = 0.01$, $\mu = 1$,$\epsilon = 0.01$, $\gamma = 0.01$, and $\lambda = 0.0001$.

A second-order fully implicit backward differencing formula (BDF-2) is used and allows us to take a larger time step for the simulation than in previous work \cite{2010HyonY_KwakD_LiuC-aa}.  For the test problem shown, 100 time steps of size $\Delta t = 0.01$ are taken.  Plots of the phase field, $\phi$, are given in figure \ref{coalesce}. Here, one can see the coalescence of the two bubbles over time and the eventual dissipation.  In addition, more refinement is done in the region of the interface.  Therefore, most of the computation is being focussed on resolving this region, which is driving the physics in the system.  

\begin{figure}[h!]
\centering
\includegraphics[scale=0.13]{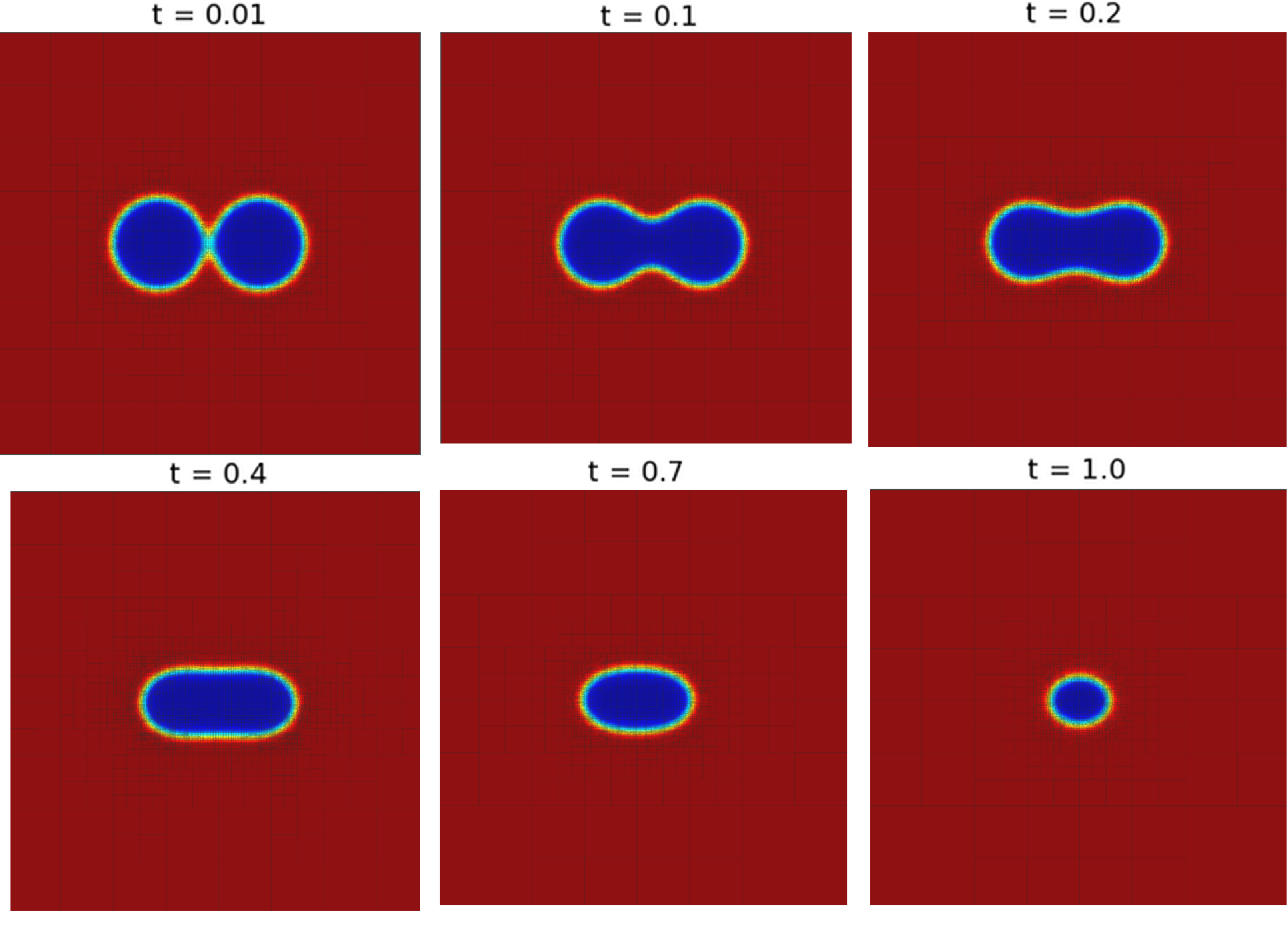}
\caption{Phase field, $\phi$, at time t = 0.01, 0.1, 0.2, 0.4, 0.7, and 1.0.  Red represents 1.0 and blue represents -1.0.}
\label{coalesce}
\end{figure}

To quantify how well the nested iteration algorithm with adaptive local refinement performs, the simulation was run using uniform refinement.  An estimate of the work is computed by calculating the number of non-zeroes in each discrete system and determining how many equivalent fine-grid matrix-vector operations are performed.  We call this measure a work unit (WU), and it is a measure of the cost of performing one relaxation sweep, such as Gauss-Seidel or Jacobi, on the finest grid of the simulation.  Table \ref{WUresult} shows that, while using adaptive refinement requires more work units (i.e. more V-cycles are needed to solve the problem), on average the number of non-zeroes on the finest grid is about one-tenth of that for uniform refinement.  This corresponds to a work ratio of about 0.21.  In other words, using adaptive local refinement saved about 80\% of the computational cost.  Therefore, on average, the above strategies on this test problem cost the equivalent of performing less than 49 Gauss-Seidel relaxation sweeps on a quadratic uniform grid with $65,536$ ($256\times256$) elements.  It should be noted here that the functional or estimate of the error for adaptive refinement is about half of the functional obtained from using uniform refinement on average.   Both schemes got below the prescribed tolerance, but, on many time steps, the adaptive algorithm got to a much lower error (occasionally by an order of magnitude).  Therefore, as far as accuracy-per-computational cost, the adaptive refinement scheme may be even better than the numbers indicate.  What we can say is that it used less than 9\% of the elements and took almost one-fifth of the computational time to solve the problem.  

\begin{table}[h!]
\centering
\begin{tabular}{|ccc|}
\hline
Avg WU&Avg Nonzeros&Avg Functional\\
\hline
&&\\
&Uniform&\\
&&\\
227.1&281,292,867&0.456\\
\hline
&&\\
&Adaptive&\\
&&\\
530.56&24,955,707&0.265\\
\hline
&&\\
Work Ratio&&Element Ratio\\
0.213&&0.089\\
\hline
\end{tabular}
\caption{Average estimates of work and accuracy using uniform refinement versus adaptive local refinement}
\label{WUresult}
\end{table}

In addition, Table \ref{NewtsCoal} shows that, for each refinement level, the number of Newton steps needed to solve the nonlinear problem decreases from almost 5 on the coarse grids to 1 on the finer grids.  For all time steps, on the finest-grid, only one Newton step was needed.  This greatly reduces the cost of the algorithm compared to if one were to only solve on the finest grid.  

\begin{table}[h!]
\centering
\begin{tabular}{|c|c|c|c}
\hline
Grid&Elements&Avg Newton Steps\\
\hline
10&4&4.6\\
9&4&2.2\\
8&16&2.9\\
7&47.1&2.7\\
6&84.9&2.3\\
5&178.1&2.0\\
4&417.7&2.0\\
3&1,005.0&1.1\\
2&2,084.9&1.0\\
1&5,812.4&1.0\\
\hline
\end{tabular}
\caption{Average Newton steps per grid level over all time steps.  Level = 1 is the finest grid.  The first grid (level 10) uses bilinear elements. }
\label{NewtsCoal}
\end{table}

Next, we see in figure \ref{coalesceE} that the energy law of this system is satisfied during the simulation.  The total conservative energy over the whole domain is decreasing and the change in energy coincides with the dissipation in the system.  Thus, the numerical simulation is resolving the expected physics of the system.

\begin{figure}[h!]
\centering
\includegraphics[scale=0.33]{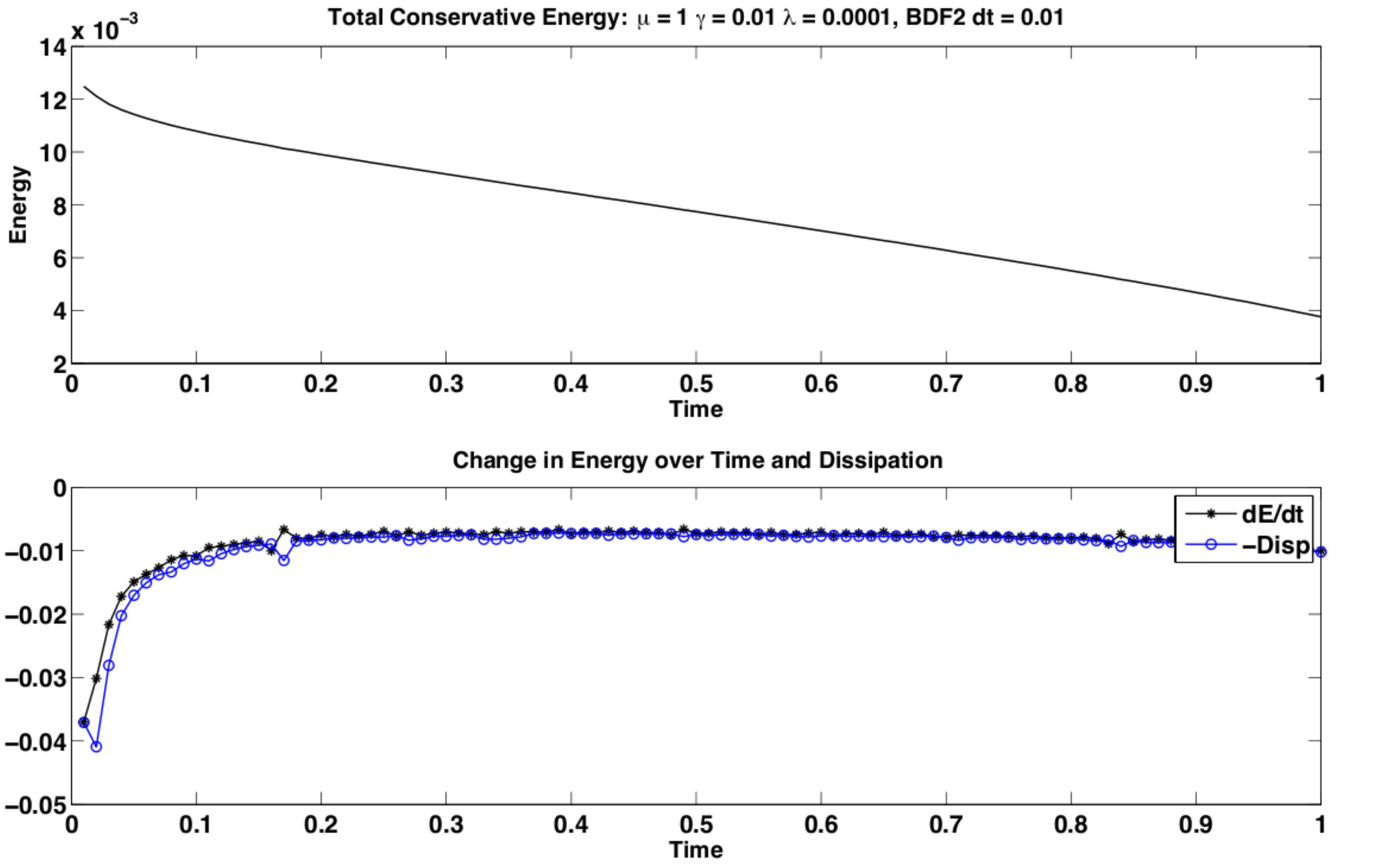}
\caption{Energy laws for coalescence problem.  Top:  Total conservative energy over time.  Bottom:  Change in energy (asterix) over time versus negative dissipation (circle)}
\label{coalesceE}
\end{figure}

\subsubsection{Square Bubble}

In this test problem, we start with a single bubble immersed in a fluid of a different phase on the unit square, $\Omega = [0,1]\times[0,1]$.  The initial interface between the two fluids is given such that the bubble is in the shape of a square.  As time evolves, the interface forces cause the bubble to transform to a circular shape before diffusing away.  In the process, oscillations between a square and diamond shape are formed.  This reflects the elastic short-time behavior of the system.  The initial conditions for this system are similar to that in \cite{2003LiuC_ShenJ-aa} and are as follows:

\begin{eqnarray*}
\vec{u} &=&\left ( \begin{array}{c}0\\0\\\end{array} \right ),\\
\phi &=& \left \{ \begin{array}{cc} +1&(x,y) \in [-\frac{1}{4},\frac{1}{4}]\times[-\frac{1}{4},\frac{1}{4}]\\-1&\mbox{otherwise}\\\end{array} \right \}.\\ 
\end{eqnarray*}
The parameters used are $\mu = 0.1$, $\epsilon = 0.02$, $\gamma = 0.01$, and $\lambda = 0.1$.

Again, a second-order fully implicit backward differencing formula (BDF-2) is used.  For the test problem shown, 100 time steps of size $\Delta t = 0.01$ are taken.  The phase field is shown in figure \ref{square}.  As in the previous example, the refinement is being placed in the appropriate regions and requires about 20\% of the work that would be required if uniform refinement were used.  

\begin{figure}[h!]
\centering
\includegraphics[scale=0.13]{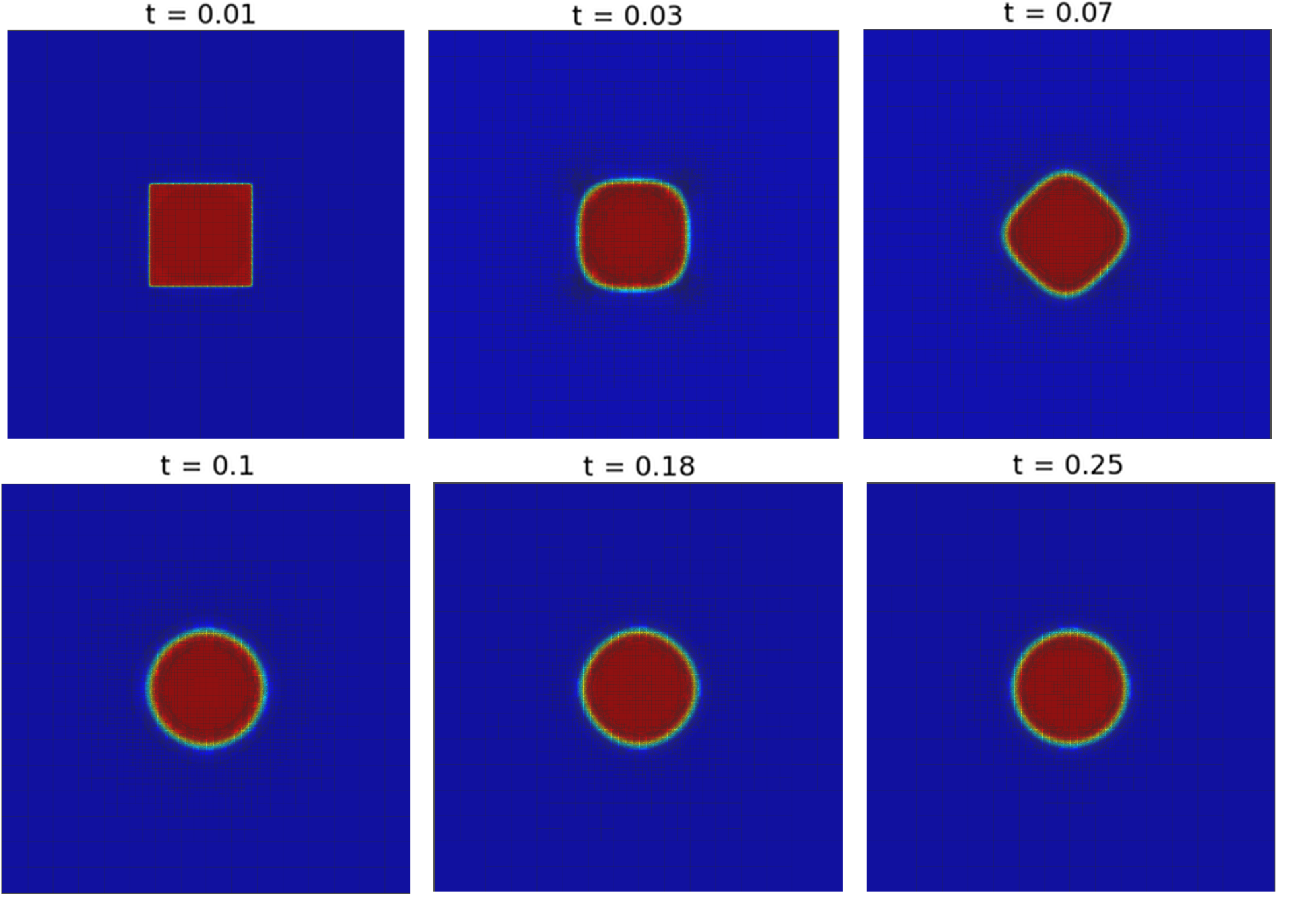}
\caption{Phase field, $\phi$, at time t = 0.01, 0.03, 0.07, 0.1, 0.18, and 0.25.  Red represents 1.0 and blue represents -1.0.}
\label{square}
\end{figure}

Due to the bigger jump at the interface, though, more Newton steps are needed in this problem than in the first.  However, using nested iteration results in many of these linearizations being performed on the coarser grids.  As an example, at time step 3, when the bubble begins to lose its square shape, the number of Newton steps goes from 10 to 1 as we move up through the grids.  See table \ref{NIresults}.

\begin{table}[h!]
\centering
\begin{tabular}{|c|c|c|c}
\hline
Grid&Elements&Newton Steps\\
\hline
9&16&10\\
8&64&4\\
7&112&5\\
6&208&5\\
5&463&5\\
4&1,423&4\\
3&3,766&3\\
2&12,793&3\\
1&40,030&1\\
\hline
\end{tabular}
\caption{Newton steps per grid level for time step 3.  Level = 1 is the finest grid.}
\label{NIresults}
\end{table}

In figure \ref{squareE}, one can see that the energy laws are being satisfied once again.  An initial increase in energy occurs before the dissipation in the system takes over, causing the energy to decay.  This is most likely due to the fact that the initial condition of the phase field is discontinuous at the interface.  Once the simulation is started, this function is immediately smoothed out and the energy behaves as expected.  By increasing $\gamma$, the elastic relaxation time in the system, the dissipation is increased and actually causes the bubble to dissipate away much faster.  In fact, the entire bubble vanishes after time step 30 and the total energy goes to zero.  This can be seen in the energy plots in figure \ref{squareE2}.  The tiny jump in energy appears to occur when the region containing the bubble finally disappears and the system is changed to that of one fluid.  With a smaller time step, this transition occurs more smoothly.  

\begin{figure}[h!]
\centering
\includegraphics[scale=0.33]{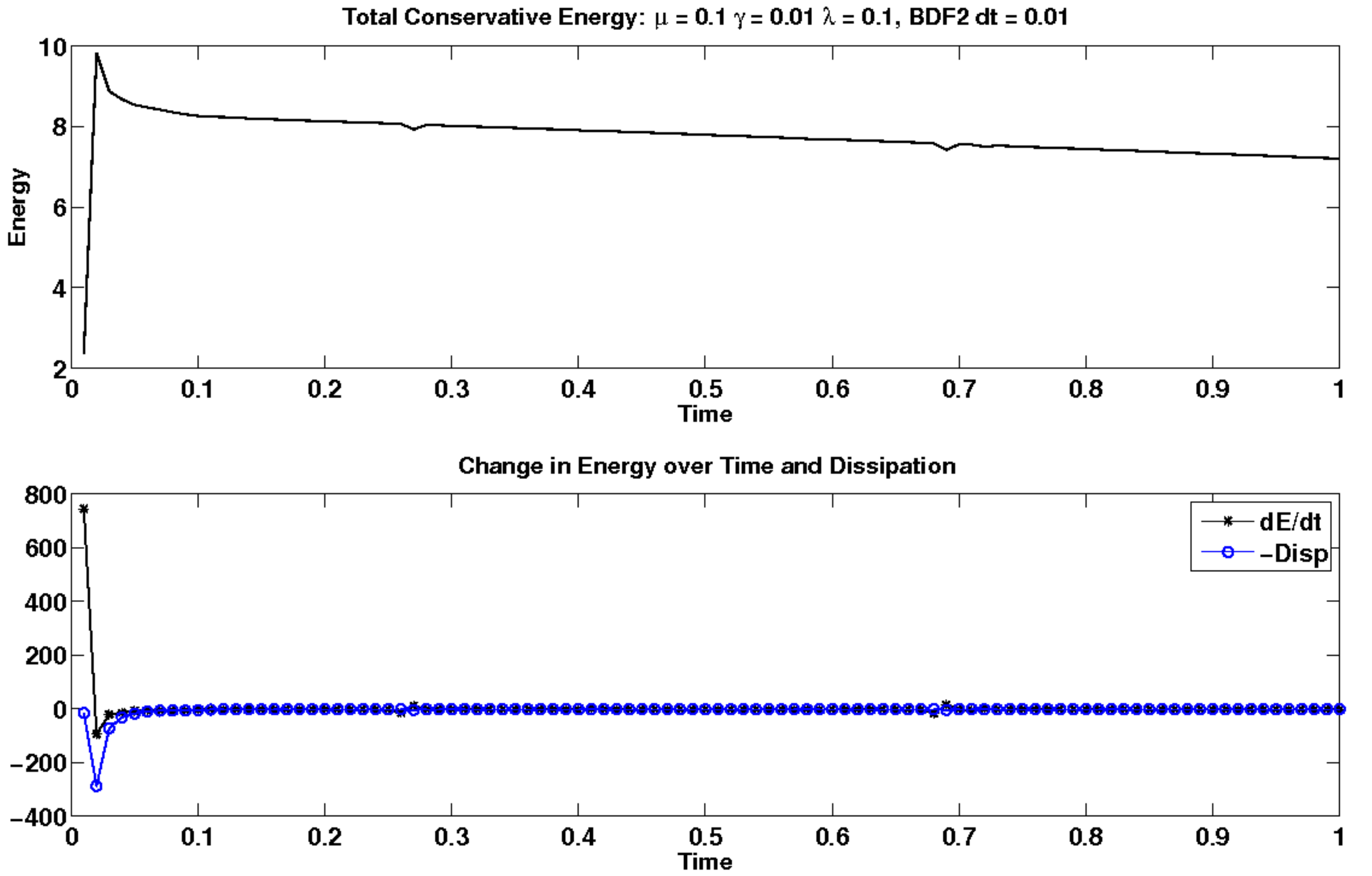}
\caption{Energy laws for square problem ($\gamma = 0.01$).  Top:  Total conservative energy over time.  Bottom:  Change in energy (asterix) over time versus negative dissipation (circle)}
\label{squareE}
\end{figure}

\begin{figure}[h!]
\centering
\includegraphics[scale=0.33]{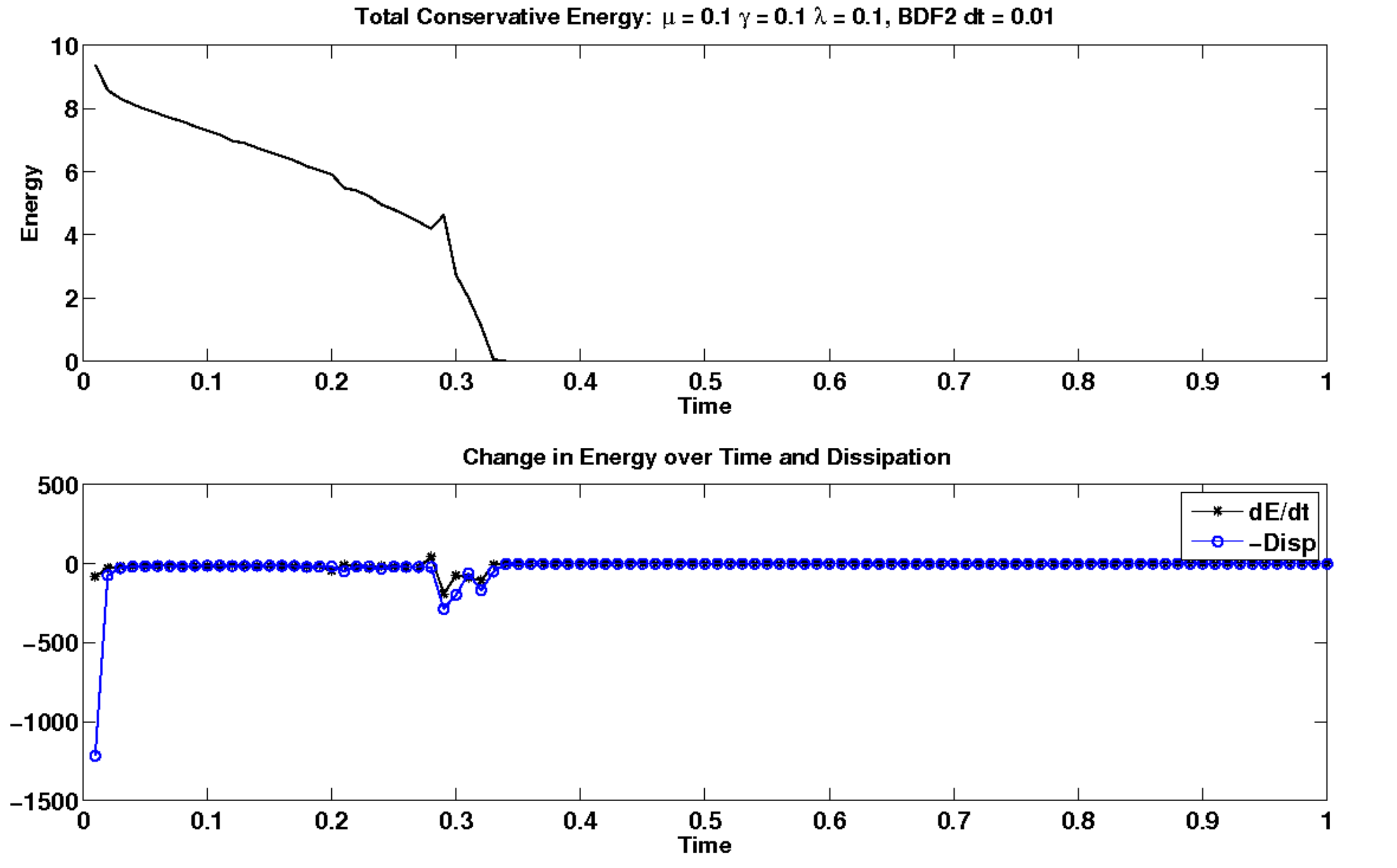}
\caption{Energy laws for square problem ($\gamma = 0.1$).  Top:  Total conservative energy over time.  Bottom:  Change in energy (asterix) over time versus negative dissipation (circle)}
\label{squareE2}
\end{figure}

\subsection{MHD}
As described above, the characteristics of the two-phase system is similar to that of resistive MHD.  The main results and description of this problem using the FOSLS and nested iteration framework can be found in \cite{2010AdlerJ_ManteuffelT_McCormickS_RugeJ-aa,2010AdlerJ_ManteuffelT_McCormickS_RugeJ_SandersG-aa,2010AdlerJ_ManteuffelT_McCormickS_NoltingJ_RugeJ_TangL-aa}.  However, here, we briefly mention the energetics of the system.  For this system, the energy law must satisfy the following equation:
$$\dfrac{d}{dt} \int_{\Omega} \left \{ \frac{1}{2}|\vec{u}|^2 + \frac{1}{2}|\vec{B}|^2 \right \} d\vec{x} = - \int_{\Omega} \left \{ \frac{1}{R_e} |\mygrad \vec{u}|^2 + \frac{1}{S_L}|\mygrad\vec{B}|^2 \right \} d\vec{x}.$$
Here, $\vec{u}$ represents the fluid velocity and $\vec{B}$ represents the magnetic field.  $R_e$ and $S_L$ are the fluid Reynolds number and magnetic Lundquist number, respectively.  

A magnetic reconnection problem is simulated using the above numerical algorithm.  Figure \ref{mhdenergy} shows that the discrete energy law is being satisfied.  The total conservative energy decays as a direct result of the dissipation in the system.  The oscillations in the change in energy plot are due to a lower order differentiation being used to compute $\frac{dE}{dt}$ after the simulation was complete.  Similar to the first two examples, it took on average of 10 WU to solve the problem using nested iteration and adaptive local refinement.  

\begin{figure}[h!]
\centering
\includegraphics[scale=0.33]{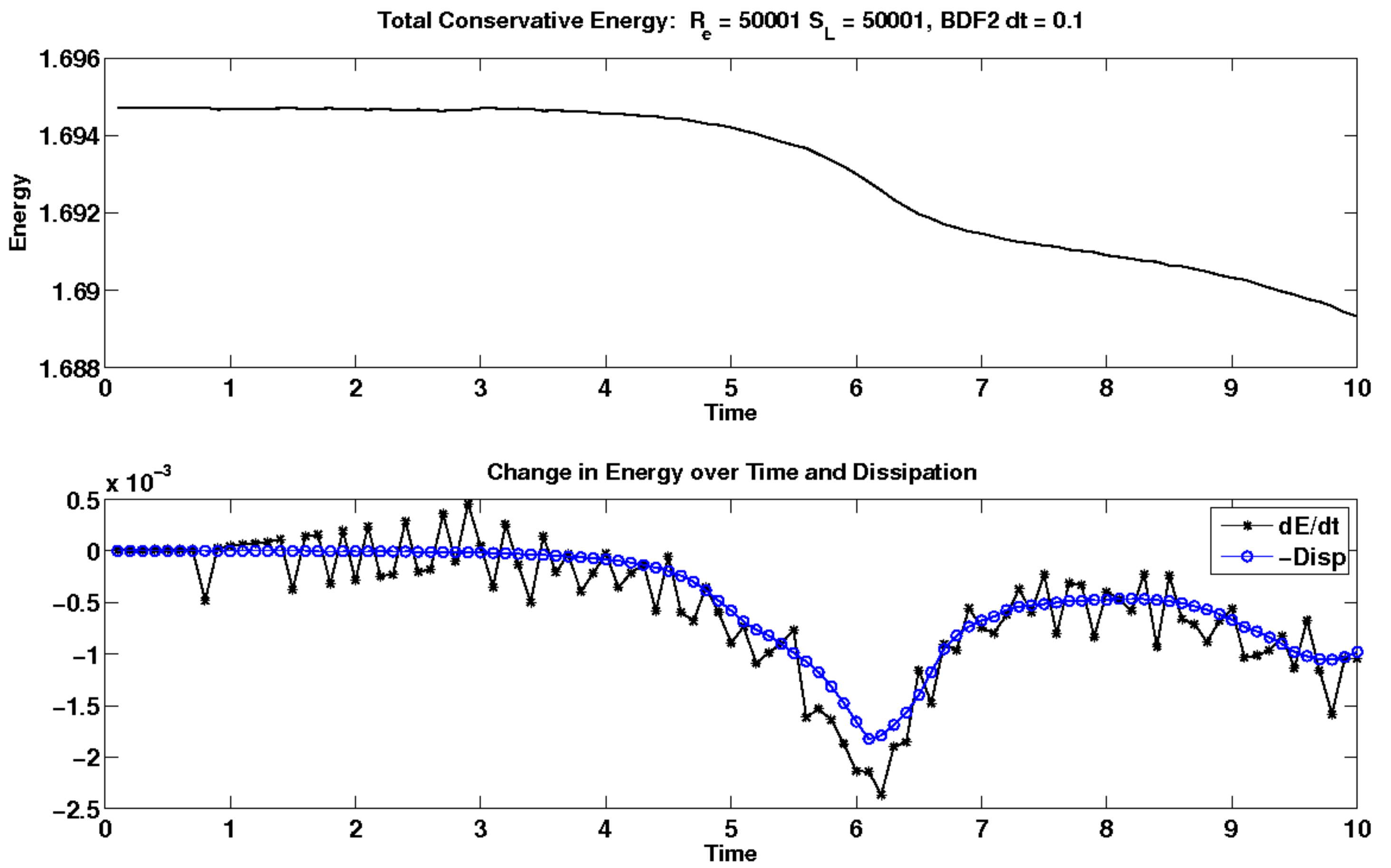}
\caption{Energy laws for mhd test problem.  Top:  Total conservative energy over time.  Bottom:  Change in energy (asterix) over time versus negative dissipation (circle)}
\label{mhdenergy}
\end{figure}


\section{Discussion}\label{conclude}  This paper shows that using the Newton-FOSLS finite-element formulation, along with nested iteration, algebraic multigrid, and adaptive local refinement, accurately resolves two-phase flow simulations while still preserving the energetics of the system.  The numerical algorithms used were designed in a general setting and have the advantage of reducing the amount of computational cost.  For the two-phase flow problems, it is possible to resolve the physics in the equivalent of less than 50 fine-grid relaxation sweeps for each time step.  This is accomplished by reducing the amount of work done on highly resolved grids and by reducing the number of elements needed to solve the problem via adaptive local refinement.  As a result, complicated systems of PDEs can be solved very efficiently without losing the ability to measure certain physical quantities.  This paper shows that, when applied to two-phase flow simulations, all the advantages of using implicit finite element methods are obtained, and the FOSLS methodology is capable of capturing the important aspects of the model, such as the energy laws.

Despite the success of the above algorithms in capturing the correct energy laws efficiently, there has been one bottleneck.  In all the simulations, classical algebraic multigrid (AMG) was used to solve the discrete linear systems.  Classical AMG was designed for scalar elliptic problems and M-matrices.  As can be seen in figure \ref{wuplot}, the average convergence factors for AMG over each time step can be quite poor when it is used on systems for which it was not designed.  While it is always bounded away from 1, one can see that the worse the AMG convergence factor gets, the more computational work is needed to solve the problem.  Therefore, future work will involve investigating various multigrid methods that will be more well-suited to the test problems described above.  Algebraic methods that are designed more specifically for systems of PDEs will be studied as well as geometric multigrid.  All the domains have been well structured and, therefore, it is possible to use more of the geometric information to create a better solver.  If it is possible to improve the linear solver using these methods, then the amount of work can be reduced even further.

\begin{figure}[h!]
\centering
\includegraphics[scale=0.43]{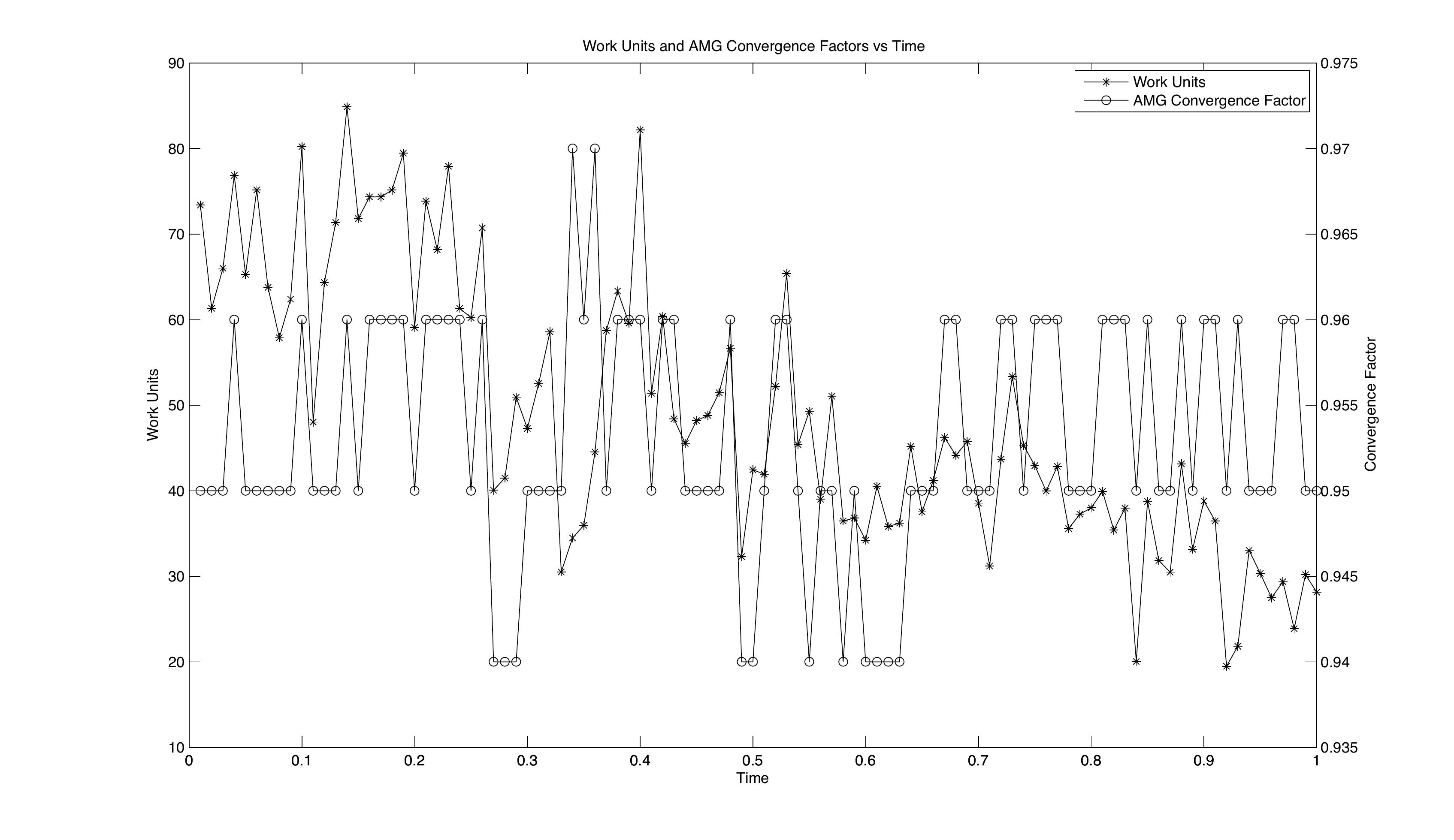}
\caption{Work units per time step as well as average convergence factor per time step.}
\label{wuplot}
\end{figure}

\section*{Acknowledgments}
This work was sponsored by the National Science Foundation under grants NSF DMS-0074299, NSF OCI-0749202, NSF DMS-0707594, EAR-0621199, OCI-0749317, DMS-0811275, and DMS-0509094, by the Department of Energy under grant numbers DE-FG02-03ER25574 and DE-FC02-06ER25784, and by Lawrence Livermore National Laboratory under contract numbers B58858.

We would also like to thank Marian Brezina, Steve McCormick, and John Ruge from the University of Colorado for their useful comments on this work.

\bibliographystyle{elsarticle-num}   
\bibliography{mybib2}        

\end{document}